\documentclass [12pt]{amsart}
\usepackage{geometry}
\usepackage{mathrsfs}
\usepackage{currvita}
\usepackage{enumerate}
\usepackage{cite}
\usepackage{xcolor}
\usepackage{graphicx}
\usepackage{amsthm}
\usepackage{amssymb}
\usepackage{enumerate}
\usepackage{tikz-cd}

\theoremstyle{plain}
\newtheorem{thm}{Theorem}[section]
\newtheorem{cor}[thm]{Corollary}
\newtheorem{lem}[thm]{Lemma}

\newtheorem{defn}[thm]{Definition}
\newtheorem{exa}[thm]{Example}
\newtheorem{rem}[thm]{Remark}

\begin{document}

\title [Intersection Graph of graded ideals]{Intersection Graph of graded ideals}

\author[{{T. Alraqad}}]{\textit{Tariq Alraqad}}

\address
{\textit{Tariq Alraqad, Department of Mathematics, University of Ha'il, Saudi Arabia.}}
\bigskip
{\email{\textit{t.alraqad@uoh.edu.sa}}}

 \author[{{H. Saber}}]{\textit{Hicham Saber}}

\address
{\textit{Hicham Saber, Department of Mathematics, University of Ha'il, Saudi Arabia.}}
\bigskip
{\email{\textit{hicham.saber7@gmail.com}}}

 \author[{{R. Abu-Dawwas }}]{\textit{Rashid Abu-Dawwas }}

\address
{\textit{Rashid Abu-Dawwas, Department of Mathematics, Yarmouk
University, Jordan.}}
\bigskip
{\email{\textit{rrashid@yu.edu.jo}}}


\date{}

\begin{abstract}
In this article we introduce and study the intersection graph of graded ideals of graded rings. The intersection graph of $G-$graded ideals of a graded ring $(R,G)$ is a simple graph, denoted by $Gr_G(R)$, whose vertices are the nontrivial graded ideals and two ideals are adjacent if they are not trivially intersected. We study graphical properties for these graphs such as connectivity, regularity, completeness, domination numbers, and girth. These intersection graphs for faithful, strong, and first strong gradings are also discussed. In addition, we investigate intersection graphs of $\mathbb{Z}_2-$graded idealization, and we deal with intersection graph of graded ideals when the grading group is an ordered groups.
\end{abstract}

 \maketitle

\section{Introduction}\label{sec1}
Throughout this article, all rings are associative with unity $1\neq 0$. Let $G$ be a group. A ring $R$, is said to be  $G$-graded if there exist additive subgroups $\{R_{\sigma}\mid \sigma\in G\}$ such that $\displaystyle R=\oplus_{\sigma\in G}R_{\sigma}$ and $R_{\sigma}R_{\tau}\subseteq R_{\sigma\tau}$ for all $\sigma,\tau\in G$. When $R$ is $G$-graded we denote that by $(R,G)$. The support of $(R,G)$ is defined as $supp(R,G)=\{\sigma\in G:R_{\sigma}\neq0\}$. If $x\in R$, then $x$ can be written uniquely as $\sum_{\sigma\in G}x_{\sigma}$, where $x_{\sigma}$ is the component of $x$ in $R_{\sigma}$. It is well known that $R_{e}$ is a subring of $R$ with $1\in R_{e}$.  A left ideal $I$ of $R$ is called $G-$graded left ideal provided that $I=\oplus_{\sigma\in G}(I\cap R_{\sigma})$. Denote by $hI^{\ast}(R)$ the set of all proper nontrivial $G-$graded left ideals of $R$. A $G-$graded left ideal is called $G-$graded maximal (resp. minimal) if it is maximal (resp. minimal) among the $G-$graded left ideals of $R$.  A left (resp. $G-$graded left) ideal of $R$ is called left (resp. $G-$graded left) essential if $I\cap J\neq \{0\}$ for all $J\in I^{\ast}(R)$ (resp. $J\in hI^{\ast}(R)$).  We call $R$, $G-$graded left Noetherian (resp. Artinian) if $R$ satisfies the ascending (resp. descending) chain condition for the $G-$graded left ideals. Analogously, $R$ is called $G-$graded local if it has a unique $G-$graded maximal left ideal. The ring $R$ is called $G-$graded domain if it is commutative and has no homogeneous nonzero zero-divisors. Similarly, $R$ is called $G-$graded division ring if every nonzero homogeneous element is a unit. A $G-$graded field is a commutative $G-$graded division ring.

In the last two decades, the theory of graded rings and modules has been receiving an increasing interest. Many authors introduced and studied, in a parallel way, the graded version of a wide range of concepts see \cite{abd16, ba17, fa11, co83, gr90, kh13, le13, ns82, re94, va81}. Another area of interest in recent years is studying graphs on algebraic structures. These studies usually aim to investigate  ring properties using graph theory concepts. Since Beck \cite{be88} introduced the concept of zero divisor graph in 1988, this approach became very popular. Total graphs, annihilating-ideal graph, and unit graphs are also examples of graphs associated to rings see \cite{as10,an08,gr90,be11}. For studies of graphs associated with graded rings in particular are somewhat rare. We only found two such studies in the literature, namely \cite{kh16, ro17}.

 In 2009, Chakrabarty et al. \cite{ch09} introduced the intersection graph of ideals of a ring. Denote by $I^{\ast}(R)$ the family of all nontrivial left ideals of a ring $R$. The intersection graph of ideals of $R$, denoted by $G(R)$ is the simple graph whose set of vertices is $I^{\ast}(R)$ where two nontrivial ideals $I$ and $J$ are adjacent if $I\cap J\neq\{0\}$. Chakrabarty et al. \cite{ch09} studied the connectivity of $G(R)$ and investigated several properties of $G(\mathbb{Z}_n)$. Akbari et al. \cite{ak13} studied these graphs more deeply. Among many results, they characterize all rings $R$ for which $G(R)$ is not connected. For other interesting studies of intersection graphs of ideals of rings the reader is referred to \cite{ab14, ab16, ak14, ja10, ja11, pu14, ra14, ra20, sa17, xu20}.

 The main theme of this work is the study of a graded version of the intersection graph of left ideals. We introduce the intersection graph of the $G-$graded left ideals of a $G-$graded ring $R$ denoted by $Gr_G(R)$.
\begin{defn}\label{d0}
Let $R$ be a $G-$graded ring. The intersection graph of the $G-$graded left ideals of $R$, denoted by $Gr_G(R)$, is the simple graph whose set of vertices is $hI^{\ast}(R)$ and two vertices $I$ and $J$ are adjacent if $I\cap J\neq\{0\}$.
\end{defn}
Sections \ref{sec2} and \ref{sec3} focus, in a way parallel to the ungraded case, on the graphical properties of $Gr_G(R)$. For these graphs, we discuss connectivity, diameter, regularity, completeness, domination numbers, and girth.
In section \ref{sec4} we study the relationship between $Gr_G(R)$ and $G(R_e)$ for some types of gradings namely left faithful, strong, and first strong gradings. In case of left $e-$faithful, we obtain an equivalence relation $\sim$  on vertices $Gr_G(R)$ by $I\sim J$ if and only if $I\cap R_e=J\cap R_e$. Then we show that the  quotient graph of $Gr_G(R)$ over the equivalence classes of $\sim$ is isomorphic to $G(R_e)$. This isomorphism allows us to extent many of the graphical properties of $G(R_e)$ to $Gr_G(R)$. We also prove that if $(R,G)$ is first strong grading then $Gr_G(R)\cong G(R_e)$. In Section \ref{sec5}, we investigate the intersection graph of graded ideals of $\mathbb{Z}_2-$graded idealizations. Section \ref{sec6} focuses on the relationship between $Gr_G(R)$ and $G(R)$ when the grading group $G$ is an ordered group.

For standard terminology and notion in graph theory, we refer the reader to the text-book \cite{bo76}. Let $\Gamma$ be a simple graph with vertex set $V(\Gamma)$ and set of edges $E(\Gamma)$. Then $|V(\Gamma)|$ is the order of $\Gamma$. If $x,y\in V(\Gamma)$ are adjacent we write that as $x-y$. The neighborhood of a vertex $x$ is $N(x)=\{y\in V(\Gamma)\mid y-x\}$ and the degree of $x$ is $deg(x)=|N(x)|$. The graph $\Gamma$ is said to be regular if all of its vertices have the same degree. A graph is called complete (resp. null) if any pair of its vertices are adjacent (res. not adjacent). A complete (resp. null) graph with $n$ vertices is denoted by $K_n$ (resp. $N_n$). A graph is called start graph if it has no cycles and has one vertex (the center) that is adjacent to all other vertices. A graph is said to be connected if any pair of its vertices is connected by a path. For any pair of vertices $x, y$ in $\Gamma$, the distance $d(x,y)$ is the length of the shortest path between them and $diam(\Gamma)$ is the supremum of $\{d(x,y)\mid x,y\in V(\Gamma)\}$. The girth of a $\Gamma$, denoted by $g(\Gamma)$ is the length of its shortest cycle. If $\Gamma$ has no cycles then $g(\Gamma)=\infty$. A graph $\Upsilon$ is a subgraph of $\Gamma$ if $V(\Upsilon)\subseteq V(\Gamma)$ and $E(\Upsilon)\subseteq E(\Gamma)$. $\Upsilon$ is called induced subgraph if any edge in $\Gamma$ that joins two vertices in $\Upsilon$ is in $\Upsilon$. A complete subgraph of $\Gamma$ is called a clique, and the order of the largest clique in $\Gamma$, denoted by $\omega(\Gamma)$, is the clique number of $\Gamma$. A dominating set in $\Gamma$ is a subset $D$ of $V(\Gamma)$ such that every vertex of $\Gamma$ is in $D$ or adjacent to a vertex in $D$. The domination number of $\Gamma$, denoted by $\gamma(\Gamma)$, is the minimum cardinality of a dominating set in $\Gamma$.

\section{Connectivity, Regularity, and diameter of $Gr_G(R)$}\label{sec2}

We start this section by stating the the following well known lemma regarding graded ideals, which will be used frequently throughout the paper.

\begin{lem}\label{B}(\cite{fa11}, Lemma 2.1) Let $R$ be a $G$-graded ring. If $I$ and $J$ are $G-$graded left ideals of $R$, then so are $I+J$ and $I\bigcap J$.
\end{lem}

The following lemma is straightforward so we omit the proof.
\begin{lem}\label{r1} Let $R$ be a $G-$graded ring and let $I$ be $G-$graded left ideal of $R$. \\
(1) $I$ is $G-$graded minimal if and only if $N(I)=\{A\in hI^{\ast}(R)\mid I\subset A\}$.\\
(2) $I$ in isolated vertex in $Gr_G(R)$ if and only if it is $G-$graded minimal as well as $G-$graded maximal.\\
(3) $I$ is $G-$graded essential if and only if $N(I)=hI^{\ast}(R)\setminus\{I\}$.
\end{lem}

The following is a well known results about $\mathbb{Z}-$graded fields (see \cite{va81}).

\begin{thm}\label{t001}Let $R$ be a commutative $\mathbb{Z}-$graded ring. Then $R$ is a $\mathbb{Z}-$graded field if and only if $R_0$ is a field and either $R=R_0$ with trivial grading or $R\cong R_0[x,x^{-1}]$ with $\mathbb{Z}-$grading $R_k=R_0x^k$.
\end{thm}




Theorem \ref{t1} gives a necessary and sufficient condition for the intersection graph of graded ideals to be disconnected. We will see that this result is analogue to the nongraded case . First we state the theorem in nongraded case.

\begin{thm}\label{t00}\cite[Corollary 2.5]{ch09}
Let $R$ be a graded ring. Then $G(R)$ is disconnected if and only if it is null graph with at least two vertices.
\end{thm}

\begin{thm}\label{t01}\cite[Corollary 2.8]{ch09}
Let $R$ be a commutative ring. Then $G(R)$ is disconnected if and only if $R$ is a direct product of two fields.
\end{thm}

\begin{thm}\label{t1}
Let $R$ be a $G-$graded ring. Then $Gr_G(R)$ is disconnected if and only if $Gr_G(R)\cong N_n$ for some $n\geq 2$.
\end{thm}

\begin{proof}
Suppose that $Gr_G(R)$ is disconnected. For a contradiction, assume $I$ and $J$ are two adjacent vertices. So $I$, $J$, and $I\cap J$ belong to the same component of $Gr_G(R)$. Since $Gr_G(R)$ is disconnected, there is a vertex $K$ that is not connected to anyone of the vertices $I$, $J$, and $I\cap J$. If $(I\cap J)+K\neq R$ then $(I\cap J)-((I\cap J)+K)-K$ is a path connecting $I\cap J$ and $K$, a contradiction. So $(I\cap J)+K=R$. Now let $a\in I$. Then $a=t+c$ for some $t\in I\cap J$ and $c\in K$. So $a-t=c\in I\cap K=\{0\}$, consequently $a=t\in I\cap J$. This implies that $I= I\cap J$. Similarly, we get $J=I\cap J$. Hence we have $I=J$ a contradiction. Therefore $Gr_G(R)$ contains no edges, and hence it is a null graph.
\end{proof}

The following result is a direct consequence of Theorem \ref{t1}.

\begin{cor}\label{c1}
Let $R$ be a $G-$graded ring. If $Gr_G(R)$ is disconnected then  $R$ contains at least two $G-$graded minimal left ideals and every $G-$graded left ideal of $R$ is principal, graded minimal, and graded maximal.
\end{cor}

\begin{thm}\label{c11}
Let $R$ be a commutative $G-$graded ring. Then $Gr_G(R)$ is disconnected if and only if $R\cong R_1\times R_2$ where $R_1$ and $R_2$ are $G-$graded fields.
\end{thm}
\begin{proof}
Assume $Gr_G(R)$ is disconnected. Then by Theorem \ref{t1} and Corollary \ref{c1}, $R$ has two $G-$graded maximal as well as $G-$graded minimal ideals $I$ and $J$ such that $I+J=R$ and $I\cap J=\{0\}$. Hence $R/I$ and $R/J$ are $G-$graded fields and $R\cong R/I\times R/J$. For the converse, assume that  $R\cong R_1\times R_2$ where $R_1$ and $R_2$ are $G-$graded fields. Then the $G-$graded ideals in $R$ are $R_1\times 0$ and $0\times R_2$. Hence $Gr_G(R)$ is disconnected.
\end{proof}

\begin{cor}\label{c101}
Let $R$ be a commutative $G-$graded ring. If $Gr_G(R)$ is connected, then every pair of $G-$graded maximal left ideals have non-trivial intersection.
\end{cor}

\begin{rem}\label{c10}
Let $R$ be a $G-$graded ring with at least two distinct $G-$graded ideals. Since $Gr_G(R)$ is a subgraph of $G(R)$, it follows that if $Gr_G(R)$ is connected then so is $G(R)$. However, the converse of this statement is not always true. Indeed, Take a field $K$ and let $R=R_1\times R_2$ where $R_1=R_2=K[x,x^{-1}]$, with $\mathbb{Z}-$grading by $(R_i)_n=Kx^{n}$, $i=1,2$. Since $R_1$ and $R_2$ are $\mathbb{Z}-$graded fields, $Gr_{\mathbb{Z}}(R)$ is disconnected. However, $R_1$ and $R_2$ are not fields, and hence $G(R)$ is connected. In fact, in light of Theorem \ref{t001}, Theorem \ref{t01}, and Corollary \ref{c11} we have the following result.
\end{rem}
\begin{cor}
Let $R$ be a commutative $Z-$graded ring such that $Gr_{\mathbb{Z}}(R)$ is disconnected. Then $R\cong R_1\times R_2$ such that one of the following is true:\\
(1) $R_1$ and $R_2$ are fields, and hence $G(R)$ is disconnected\\
(2) either $R_1$ or $R_2$ is isomorphic to $K[x,x^{-1}]$ for some field(s) $K$. Consequently $G(R)$ is connected
\end{cor}
\begin{thm}\label{t2}
Let $R$ be a $G-$graded ring. If $Gr_G(R)$ is connected then $diam(Gr_G(R)) \leq 2$.
\end{thm}

\begin{proof}
Let $I,J$ be two vertices in $Gr_G(R)$. If $I\cap J\neq \{0\}$ then $d(I,J)=1$. Suppose  $I\cap J= \{0\}$.
If there exits a $G-$graded left ideal $K\subseteq I$ such that $K+J\neq R$, then $I-(K+J)-J$ is a path, and hence $d(I,J)=2$. So we may assume $K+J=R$ for every  $G-$graded left ideal $K\subseteq I$. Now we show that $I$ is $G-$graded minimal. Let $K\subseteq I$ be a $G-$graded left ideal, and let $x\in I$. Then $x=y+b$ for some $y\in I$ and $b\in J$. So we have $x-y=b\in I\cap J=\{0\}$, and hence $x=y\in K$. Consequently $I=K$. Therefore $I$ is $G-$graded minimal. Since $Gr_G(R)$ is connected, by Lemma \ref{r1}, $I$ is not $G-$graded maximal, and so there exists $G-$graded left ideal $Y$ such that $I\subsetneq Y$. Assume $Y\cap J=\{0\}$. Let $y\in Y$ then $y=a+b$ for some $a\in I$ and $b\in J$. Hence $y-a=b\in Y\cap J=\{0\}$, which yields $y=a$, and hence $Y=I$, a contradiction. So $Y\cap J\neq\{0\}$. Hence $I-Y-J$ is a path. Therefore $d(I,J)\leq 2$. This completes the proof.
\end{proof}

\begin{thm}\label{t51}
Let $R$ be a commutative $G-$graded ring. Then $R$ is $G-$graded domain if and only if $R$ is $G-$graded reduced and $Gr_G(R)$ is complete.
\end{thm}

\begin{proof}
Suppose $R$ is $G-$graded domain. Then clearly $R$ is $G-$graded reduced. Now, let $I,J\in hI^{\ast}(R)$, and take $0\neq a\in I\cap h(R)$ and $0\neq b\in J\cap h(R)$. Then $0\neq ab\in I\cap J$. Hence $I$ and $J$ are adjacent. This proves that $Gr_G(R)$ is complete.  Conversely, suppose that $R$ is $G-$graded reduced and $Gr_G(R)$ is complete. Assume that there are $a,b\in h(R)\setminus\{0\}$ such that $ab=0$. Since $Gr_G(R)$ is complete then there exists $0\neq c\in \left\langle a \right\rangle \cap\left\langle b\right\rangle\cap h(R)$. Hence $c^2\in  \left\langle a\right\rangle\left\langle b\right\rangle=\{0\}$. This implies that $c^2=0$, a contradiction. Therefore $R$ is $G-$graded domain.
\end{proof}

\begin{thm}\label{t52}
If $R$ is a $G-$graded Artinian ring such that $Gr_G(R)$ is not null graph, then  the followings are equivalent:\\
(1) $Gr_G(R)$ is regular\\
(2) $R$ contains a unique $G-$graded minimal ideal.\\
(2) $Gr_G(R)$ is complete.
\end{thm}
\begin{proof}
$(1)\Rightarrow (2)$
Suppose $Gr_G(R)$ is regular. Seeking a contradiction, assume that $R$ contains two distinct $G-$graded minimal ideals $I$ and $J$. Then $I$ and $J$ are nonadjacent. Since $d(I,J)\leq 2$, there is a $G-$graded left ideal $K$ that adjacent to both $I$ and $J$. Hence by minimality of $I$, we get $I\subseteq K$. This implies that $N(I)\subset N(K)$, consequently $deg(K)>deg(I)$, a contradiction. Hence $R$ contains a unique $G-$graded minimal ideal.

$(2)\Rightarrow (3)$
Suppose $R$ contains a unique $G-$graded minimal left ideal, say $I$. Let $J$ and $K$ be two $G-$graded left ideals in $R$. Since $R$ is a $G-$graded Artinian,  we have $I\subseteq J$ and $I\subseteq K$, and so $J$ and $K$ are adjacent. Therefore $Gr(R)$ is complete.

$(2)\Rightarrow (3)$ Straightforward

\end{proof}
\section{Domination, clique, and girth of $Gr_G(R)$}\label{sec3}
A commutative $G-$graded ring $R$ is called $G-$graded decomposable if there is a pair of nontrivial $G-$graded ideals $S$ and $T$ of $R$, such that  $R\cong S\times T$. If $R$ is not $G-$graded decomposable then it is called $G-$graded indecomposable.

\begin{lem}
Let $S$ and $T$ be commutative $G-$graded rings with $1$. Then $I$ is a $G-$graded ideal of $S\times T$ if and only if there are $G-$graded ideals $I_s$ of $S$ and $I_t$ of $T$ such that $I=I_s\times I_t$.
\end{lem}

\begin{thm}\label{t6} Let $R$ be commutative $G-$graded ring. Then $\gamma(Gr_G(R))\leq 2$. Furthermore the followings are true.\\
(1) If $R$ is $G-$graded indecomposable then  $\gamma(Gr_G(R))=1$.\\
(2) If $R\cong S\times T$ for some nontrivial graded ideals $S,T$ of $R$ then $\gamma(Gr_G(R))= 2$ if and only if $\gamma(Gr_G(S))= \gamma(Gr_G(T))= 2$.
\end{thm}

\begin{proof}
If  $R\cong S\times T$ for some nontrivial graded ideals $S$ and $T$, then the $\{S\times \{0\}, \{0\}\times T\}$ is a dominating set, and hence $\gamma(Gr_G(R))\leq 2$.  Suppose that $R$ is $G-$graded indecomposable. Let $M$ be a $G-$graded maximal left ideal of $R$. If there exists $J\in hI^{\ast}(R)$ such that $M\cap J=\{0\}$ then $M+J=R$ and hence $R\cong M\times J$, a contradiction. So $M\cap J\neq\{0\}$ for all $J\in hI^{\ast}(R)$. Consequently  $\{M\}$ is a dominating set, and hence $\gamma(Gr_G(R))=2.$ Now suppose $R\cong S\times T$ for some nontrivial $G-$graded ideals $S$ and $T$. It is straightforward to show that $\{I\times J\}$ is a dominating set in $Gr_G(R)$ if and only if $\{I\}$ is dominating set in $Gr_G(S)$ or $\{J\}$ is dominating set in $Gr_G(T)$. This completes the proof.
\end{proof}

\begin{lem}\label{l18}
Let $R$ be a $G-$graded ring. If $\omega(Gr_G(R)) < \infty$, then R is $G-$graded Artinian.
\end{lem}

\begin{proof}
Let $I_1\supseteq I_{2}\supseteq \cdots I_{n} \cdots $ be a descending chain of $G-$graded left ideals. Then $\{I_k\}_{k=1}^{\infty}$ is a clique in $Gr_G(R)$, and hence it is finite.
\end{proof}

\begin{thm}\label{l187}
Let $R$ be a commutative $G-$graded ring. Then \\
(1) $\omega(Gr_G(R))=1$ if and only if $Gr_G(R)= N_1$ or $N_2$\\
(2) If $1<\omega(Gr_G(R))< \infty$ then the number of $G-$graded maximal left ideals of $R$ is finite.
\end{thm}
\begin{proof}
(1) Suppose $\omega(Gr_G(R))=1$.  Assume $|Gr_G(R)|\geq 2$. Then $Gr_G(R)$ is disconnected. So, by Corollary \ref{c11},  $R$ is a direct product of two $G-$graded fields, consequently $Gr_G(R)=N_2$. The converse is clear.

(2) Suppose $1<\omega(Gr_G(R))< \infty$. So $Gr_G(R)$ is connected. Then, by Corollary \ref{c101}, the set of $G-$graded maximal left ideals of $R$ forms a clique, and hence it is finite.
\end{proof}

\begin{thm}\label{t3}
If $R$ is a $G-$graded ring then $gr(Gr_G(R))=\{3,\infty\}$
\end{thm}

\begin{proof}
Assume $gr(Gr_G(R))$ is finite and let $I_0-I_1-\cdots -I_n$ be a cycle. If $I_0\cap I_1=I_0$ then $I_n-I_0-I_1$ is a $3-$cycle. Similarly, if $I_0\cap I_1=I_1$ then $I_0-I_1-I_2$ is a $3-$cycle. The remaining case is that $I_0\cap I_1\neq I_0$ or $I_1$. In this case we obtain the $3-$cycle $I_0-I_1-(I_0\cap I_1)$. Hence $gr(Gr_G(R))=3.$
\end{proof}

In the next theorem we characterize $G-$graded rings $R$ such that $g(Gr_G(R))=\infty$. In fact, this result can be refer to as the graded version of \cite[Theorem 17]{ak13}.
\begin{thm}\label{t4}
Let $R$ be a $G-$graded ring such that $Gr_G(R)$ is not a null graph. If $gr(Gr_G(R))=\infty$ then  $R$ is a $G-$graded local ring and $Gr_G(R)$ is a star whose center is the unique $G-$graded maximal left ideal of $R$, say $M$. Moreover, one of the followings hold:\\
(1) $M$ is principal. In this case $Gr_G(R)=K_1$ or $K_2$.\\
(2) The minimal generating set of homogeneous elements of $M$ has size $2$. In this case $M^2=\{0\}$.
\end{thm}
\begin{proof}
Suppose $M_1$ and $M_2$ are two distinct $G-$graded maximal left ideals of $R$. Then by Theorem \ref{t2}, $d(M_1,M_2)\leq 2$. If $M_1 \cap M_2\neq\{0\}$ then $M_1-(M_1\cap M_2 )-M_2$ is a $3-$cycle, a contradiction. Suppose $M_1 \cap M_2=\{0\}$. Then by Theorem \ref{t2}, there exists a $G-$graded left ideal $I$ that is adjacent to both $M_1$ and $M_2$. Since $M_1 \cap M_2=\{0\}$, $I\not\subseteq M_1$ and $I\not\subseteq M_2$. So $I-M_1-M_2$ is a $3-$cycle in $Gr_G(R)$, a contradiction. Hence $R$ has a unique $G-$graded maximal ideal, and hence it is $gr-$local ring. Let $M$ be the $G-$graded maximal left ideal and suppose $M\cap J=\{0\}$ for some $J\in hI^{\ast}(R)$. Then $M\subsetneq M+J$, and hence $M+J=R$. So $M$ is $G-$graded maximal as well as $G-$graded minimal, which implies $Gr_G(R)$ is null graph, a contradiction. So $M\cap J\neq\{0\}$ for all $J\in hI^{\ast}(R)$. Moreover, since $Gr_G(R)$ has no cycles then $J\subseteq M$ for all $J\in hI^{\ast}(R)$. So fare we proved that $Gr_G(R)$ is a star whose center is $M$. Now we proceed to prove parts (1) and (2).
Since $R$ is $G-$graded Artinian, by  \cite[Corollary 2.9.7]{Nastasescue} $R$ is $G-$graded Notherian. So $M$ is generated by a finite set of homogeneous elements. If a minimal set of homogeneous generators has at least three elements, containing say $a,b,c,...$, then $M-(Ra+Rb)-(Rb+Rc)$ is a $3-$cycle in $Gr_G(R)$, a contradiction. So a minimal set of homogeneous generators of $M$ has at most two elements. Moreover, since $M$ is finitely generated and $J^g(R)=M$ (where $J^g(R)$ the graded Jacobson radical of $R$), by \cite[Corollary 2.9.2]{Nastasescue} $M\supsetneq M^2 \supsetneq M^3 \supsetneq \cdots $. In addition, since $Gr_G(R)$ has no $3-$cycles, we get $M^3=0$.\\
Case 1: Suppose $M=Ra$ for some $a\in h(R)$. Let $I\in hI^{\ast}(R)$ and let $x\in I\cap h(R)$. Then $x=ya$ for some $y\in R$. Since $x,a \in h(R)$, it results that $y\in h(R)$. If $y\notin M$, $Ry=R$, because $M$ is the only $G-$graded maximal left ideal. So $y$ is a unit, and hence $I=M$. Assume $y\in M$. Then, we get $x=wa^2$ for some $w\in h(R)$. Similarly, if $w\notin M$, then $I=Ra^2$, otherwise $I=Ra^3$, a contradiction. Therefore we have that if $Ra^2=0$ then $Gr_G(R)=K_1$, otherwise $Gr_G(R)=K_2$.\\
Case2: Assume the minimal set of homogeneous generators  of $M$ has two elements say $a,b$ i.e $M=Ra+Rb$. Since $Gr_G(R)$ has no $3-$cycles,  $Ra$ and $Rb$ are $G-$graded minimal. Moreover, we have $Ra$ and $Rb$ are left subideals of $J^g(R)$. By \cite[Corollary 2.9.2]{Nastasescue} it results that $(Ra)^2=RaRb=RbRa=(Rb)^2=0$, and hence $M^2=0$.
\end{proof}

\section{Intersections graph of types of gradings}\label{sec4}
Note that if $I_e$ is left ideal of $R_e$ then $RI_e$ is a $G-$graded left ideal of $R$. Moreover, $RI_e\cap R_e=I_e$.


\begin{thm}\label{t100}
Let $R$ be a $G-$graded ring such that $R_e$ contains at least two proper left ideals. If $G(R_e)$ is connected then $Gr_G(R)$ is connected, and hence $G(R)$ is connected.
\end{thm}
\begin{proof}
Since $G(R_e)$ is connected then it must contain an edge. Let $I_e$ , $J_e$ be two adjacent vertices of $G(R_e)$. Then $RI_e$ and $RJ_e$ are vertices in $Gr_G(R)$. Moreover $RI_e\cap R_e=I_e$ and $RJ_e\cap R_e=J_e$, and so $RI_e\neq RJ_e$. Additionally, we have $\{0\}\neq I_e\cap J_e\subseteq RI_e\cap RJ_e$. Therefore $Gr_G(R)$ is not null, and hence it is connected.
\end{proof}

\begin{rem} The converse of Theorem \ref{t100} need not to be true. Indeed, choose $R_e$ to be any ring that has two minimal left ideals $I$ and $J$ (for instance the ring in \cite[Example 2.6]{ch09}).  The ring of polynomials, $R=R_e[x]$ is $\mathbb{Z}-$graded by the grading $R_k=R_ex^k$, $k\geq0$ and $R_k=0$, $k<0$. The ideals $Rx$ and $Rx^2$ are adjacent in $Gr_G(R)$ and so $Gr_G(R)$ is connected, while $G(R_e)$ is disconnected. Later we show that the converse is true whenever the grading is left $e-$faithful.
\end{rem}

A grading $(R,G)$ is called left $\sigma-$faithful for some $\sigma \in G$, if $R_{\sigma\tau^{-1}}x_{\tau}\neq\{0\}$ for every $\tau\in G$, and every nonzero $x_{\tau}\in R_{\tau}$. If $(R,G)$ is left $\sigma-$faithful for all $\sigma\in G$ then it is called left faithful.

\begin{lem}\label{51}
A grading $(R,G)$ is left $\sigma-$faithful for some $\sigma \in G$ if and only if $I\cap R_{\sigma}\neq\{0\}$ for all $I\in hI^{\ast}(R)$
\end{lem}
\begin{proof}
Suppose $(R,G)$ is left $\sigma-$faithful for some $\sigma \in G$. Let $I\in hI^{\ast}(R)$ and take a nonzero element $x_{\tau}\in I\cap R_{\tau}$ for some $\tau\in G$. Then $R_{\sigma\tau^{-1}}x_{\tau}\neq\{0\}$. So we have $\{0\}\neq R_{\sigma\tau^{-1}}x_{\tau}\subseteq R_{\sigma\tau^{-1}}R_{\tau}\subseteq R_{\sigma\tau^{-1}\tau}=R_{\sigma}$. On the other hand $R_{\sigma\tau^{-1}}x_{\tau}\subseteq I$. Therefore $I\cap R_{\sigma}\neq\{0\}$. Conversely, assume $I\cap R_{\sigma}\neq\{0\}$ for all $I\in hI^{\ast}(R)$. If $x_{\tau}$ is a nonzero homogenous element of degree $\tau$, for some  $\tau \in G$, then $Rx_{\tau}\in hI^{\ast}(R)$. So by assumption, $Rx_{\tau}\cap R_{\sigma}\neq\{0\}$. Since $R_{\rho}x_{\tau}\subseteq R_{\rho\tau}$ for each $\rho \in G$, we get $R_{\rho}x_{\tau}\cap R_{\sigma}=\{0\}$ for all $\rho \in G\setminus\{\sigma\tau^{-1}\}$. This implies that $R_{\sigma\tau^{-1}}x_{\tau}\cap R_{\sigma}\neq\{0\}$, consequently $R_{\sigma\tau^{-1}}x_{\tau}\neq\{0\}$. Therefore $(R,G)$ is left $\sigma-$faithful.
\end{proof}

Let $(R,G)$ be left $e-$faithful grading. By Lemma \ref{51} we have $I\cap R_e\neq\{0\}$ for all $I\in hI^{\ast}(R)$. Define a relation $\sim$ on the vertices of $Gr_G(R)$ by  $I\sim J$ if and only if $I\cap R_e=J\cap R_e$. Clearly $\sim$ is an equivalence relation on $hI^{\ast}(R)$. The classes of $\sim $ are $\{[RI_e]\mid I_e\in I^{\ast}(R_e)\}$. These classes satisfy the followings.\\
(1) For each $I_e\in I^{\ast}(R_e)$, $[RI_e]$ is a clique in $Gr_G(R)$.\\
(2) If $K\in [RI_e]$ and $L\in [RJ_e]$ then $K\cap L\neq\{0\}$ if and only if $I_e\cap J_e\neq\{0\}$. To see this, note that by Lemma \ref{51} $K\cap L\neq{0}$ if and only if $K\cap L\cap R_e\neq\{0\}$. Since $K\cap R_e=I_e$ and $L\cap R_e\neq\{0\}$ we get $K\cap L\neq\{0\}$ if and only if $I_e\cap J_e\neq\{0\}$.

Define a graph $Gr_e(R)$ on the classes of $\sim$ where $[K]$ and $[L]$ are adjacent if and only if $K\cap L\neq\{0\}$. This adjacency operation is well defined by (2) above. In fact $Gr_e(R)$ is the quotient graph of $Gr_G(R)$ over the classes of $\sim$.

\begin{thm}\label{t1001}
Let $(R,G)$ be left $e-$faithful grading. Then the map $\phi:G(R_e)\longrightarrow Gr_e(R)$ defined by $\phi(I_e)=[RI_e]$ is a graph  isomorphism.
\end{thm}
\begin{proof}
Let $I_e,J_e\in I^{\ast}(R_e)$.  Since $I_e=RI_e\cap R_e$ and $J_e=RJ_e\cap R_e$, it follows that $I_e\neq J_e$ if and only if $[RI_e]\neq [RJ_e]$. Hence $\phi$ is a set bijection. Additionally from (2) above we have $I_e\cap J_e\neq\{0\}$ if and only if $RI_e\cap RJ_e\neq \{0\}$. Therefore $\phi$ is a graph isomorphism.
\end{proof}

\begin{thm}
Let  $(R,G)$ be left $e-$faithful. Then $G(R_e)$ is connected if and only if $Gr_G(R)$ is connected.
\end{thm}
\begin{proof}
The ``if'' part is  Theorem \ref{t100}. For the ``only if'' part, assume $Gr_G(R)$ is connected and let $I_e, J_e$ be two distinct vertices in $G(R_e)$. If $RI_e \cap RJ_e\neq\{0\}$, then by Theorem \ref{t1001} $I_e\cap J_e\neq\{0\}$, and hence $I_e-J_e$ is a path. Assume $RI_e \cap RJ_e=\{0\}$. By Theorem  \ref{t2}, there is $K\in hI^{\ast}(R)$ such that $RI_e\cap K\neq\{0\}$ and $RI_e\cap K\neq\{0\}$. Then $RI_e\cap K\cap R_e\neq\{0\}$ and $RI_e\cap K\cap R_e\neq\{0\}$, consequently $I_e\cap (K\cap R_e)$ and $J_e\cap (K\cap R_e)$ are nontrivial. Hence we obtain a path connecting $I_e$ and $J_e$ in $G(R)$. Therefore $G(R)$ is connected.
\end{proof}

\begin{cor}
Let $(R,G)$ be $e-$faithful grading where $R$ is a commutative. Then $R_e$ is direct product of two fields if and only if $R$ is direct product of two $G-$graded fields.
\end{cor}
\begin{proof}
The proof follows directly from Theorem \ref{t01} and Corollary \ref{c11}.
\end{proof}

\begin{thm}
Let $(R,G)$ be $e-$faithful grading. Then $\gamma(G(R_e))=\gamma(Gr_G(R))$.
\end{thm}
\begin{proof}
Let $S\subseteq I^{\ast}(R_e)$ be a minimal dominating set in $G(R_e)$, and let $\mathscr{S}=\{RI_e\mid I_e\in S\}$. By Theorem \ref{t1001} we have $|\mathscr{S}|=|S|$, and since $[RI_e]$ is a clique in $Gr_G(R)$, we get $\mathscr{S}$ is a dominating set in $Gr_G(R)$. Hence $\gamma(G(R_e))\geq \gamma(Gr_G(R)$. Now assume $\mathscr{S}$ is a minimal  dominating set in $\gamma(Gr_G(R)$, and let $S=\{I\cap R_e\mid I\in \mathscr{S}\}$. So $S$ is a dominating set in $G(R_e)$. If $[I]=[J]$ for some $I,J\in \mathscr{S}$ with $I\neq J$, then $S\setminus\{I\}$ is a dominating set in $Gr_G(R)$, a contradiction. Hence $|S|=|\mathscr{S}|$. So $\gamma(G(R_e))\leq \gamma(Gr_G(R))$.
\end{proof}

\begin{cor}
Let $(R,G)$ be $e-$faithful grading. Then $\omega (Gr_G(R))<\infty$ if and only if $\omega (G(R_e))<\infty$ and $|[RI_e]|\leq \infty$ for all $I_e\in I^{\ast}(R_e)$. Moreover, if $\omega (Gr_G(R))<\infty$ then $\omega (Gr_G(R))=Max\left\{\displaystyle\sum_{I_e\in C}|[RI_e]|\mid C \text{is a clique in } G(R_e)\right\}$.
\end{cor}
\begin{proof}
It is clear that $C$ is a clique in $G(R)$ if and only if $\bigcup_{I_e\in C}[I_e]$ is a clique in $Gr_G(R)$. Hence the result.
\end{proof}
A grading $(R,G)$ is called strong (resp. first strong) if $1\in R_{\sigma}R_{\sigma^{-1}}$ for all $\sigma\in G$ (resp. $\sigma\in supp(R,G)$) (see \cite{abd10,ns82,re94}). It is know that $(R,G)$ is strong if and only if $R_{\tau}R_{\sigma}=R_{\tau\sigma}$ for all $\tau,\sigma\in G$. In \cite[Corollary 1.4]{ns82} it is proven that if $(R,G)$ is a strong grading and $I$ is a left $G-$graded ideal of $R$, then $I=RI_e$, where $I_e=I\cap R_e$. In fact this result is still true in case $H=supp(R,G)$ is a subgroup of $G$ and  $R=\oplus_{\sigma\in H}R_{\sigma}$ is a strongly $H-$graded ring. Fact 2.5 in \cite{re94} states that $(R,G)$ is first strong if and only if $H=supp(R,G)\leq G$ and $(R,H)$ is strong. So next we state a weaker version of \cite[Corollary 1.4]{ns82}.

\begin{lem}\label{l0} Let $(R,G)$ be first strong grading. Then for every $I\in hI^{\ast}(R)$, $I=RI_e$, where $I_e=I\cap R_e$
\end{lem}

\begin{thm}\label{56} Let $(R,G)$ be first strong grading. Then $G(R_e)\cong Gr_G(R)$ \end{thm}

\begin{proof}
Since $(R,G)$ is first strong , so by Lemma \ref{l0}, we have $hI^{\ast}(R)=\{RI_e\mid I_e\in I^{\ast}(R_e)\}$. Moreover $(R,G)$ is left $e$-faithful, because if for some $\tau \in supp(R,G)$ and $x_{\tau}\in R_{\tau}$, we have $R_{\tau^{-1}}x_{\tau}=\{0\}$, then $R_ex_{\tau}=R_{\tau}R_{\tau^{-1}}x_{\tau}=\{0\}$, and hence $x_{\tau}=0$. Now the result follows by Theorem \ref{t1001}.
\end{proof}




\begin{exa}
Let $R$ be a ring and $G$ be a finite group then the group ring $R[G]$ is strongly $G-$graded ring by $(R[G])_{\sigma}=R\sigma$. Hence by Theorem \ref{56}, $Gr_G(R[G])\cong G(R)$.\\
\end{exa}

\section{Intersection graph of graded ideals of idealization}\label{sec5}
Let $R$ be a commutative ring and $M$ be an $R-$module. Then the idealization $R(+)M$ is the ring whose elements are those of $R\times M$ equipped with addition and multiplication defined by $(r,m)+(r',m')=(r+r',m+m')$ and $(r,m)(r',m')=(rr',rm'+r'm)$ respectively. The idealization $R(+)M$ is $\mathbb{Z}_2-$graded by the gradation $(R(+)M)_0=R\oplus 0$ and $(R(+)M)_1=0\oplus M$. However this grading is not first-strong and not left $e-$faithful because $(0\oplus M)^2=0\oplus0\neq R\oplus 0$. In the sequel we assume that $M\neq0$ $R(+)M\cong R$ and the $Z_2-$grading of $R(+)M$ is given by $(R(+)M)_0=R\oplus 0$ and $(R(+)M)_1=0\oplus M$. The next theorem due to \cite{an09} gives a characterization of the $\mathbb{Z}_2-$graded ideals of $R(+)M$.

\begin{lem}\label{17}\cite[Theorem 3.3]{an09}
Let $R$ be a commutative ring and $M$ be an $R-$module. Then\\
(1) The $Z_2-$graded ideals of $(R(+)M)$ have the form $I(+)N$ weher $I$ is an ideal of $R$, $N$ is a submodule of $M$ and $IM\subset N$.\\
(2) If $I_1(+)N_1$ and $I_2(+)N_2$ are $Z_2-$graded ideals of $R(+)M$ then $(I_1(+)N_1) \cap (I_2(+)N_2)=(I_1\cap I_2) (+) (N_1\cap N_2)$.
\end{lem}

\begin{thm}\label{t777}
Let $R$ be a commutative ring and $M$ be an $R-$module. Then\\
(1) $Gr_{Z_2}(R(+)M)$ is disconnected if and only if $R$ is a field and $M$ is a simple module.\\
(2) If one of the followings holds then $g(Gr{Z_2}(R(+)M))=3$.
\begin{enumerate}[(i)]
\item\label{71} $R$ and $M$ are both not simple.
\item\label{72} $|G(R)|\geq 2$.
\item\label{73} $RM\neq M$
\end{enumerate}
\end{thm}
\begin{proof}
(1) Suppose $Gr_{Z_2}(R(+)M)$ is disconnected. If $I$ is a nontrivial proper ideal of $R$ then $I(+)M$ and $0(+)M$ are adjacent in $Gr_{Z_2}(R(+)M)$, a contradiction. So $R$ is simple. Similarly, if $N$ is a nontrivial proper submodule of $M$ then  $0(+)M$ and $0(+)N$ are adjacent in $Gr_{Z_2}(R(+)M)$, a contradiction. So $M$ is simple. Conversely, assume $R$ and $M$ are simple. Then the $Z_2$-graded proper ideals of $R(+)M$ are $0(+)M$ and possibly $R(+)0$ (if $Ann_R(M)=0$). In either case $Gr_{Z_2}(R(+)M)$ is disconnected.

(2) (\ref{71}) Let $I$ be a nontrivial proper ideal of $R$ and $N$ be a nontrivial proper submodule of $M$. Then $I(+)M-0(+)M-0(+)N$ is a $3-$cycle in $Gr_{Z_2}(R(+)M)$. Hence $g(Gr{Z_2}(R(+)M))=3$.  \\
(\ref{72}) Suppose $I$ and $J$ be distinct nontrivial proper ideal of $R$. Then $I(+)M-J(+)M-0(+)M$ is a $3-$cycle. Hence the result.\\
(\ref{73}) Suppose $RM\neq 0$, then $(R(+)RM)-(0(+)RM)-(0(+)M)$ is a $3-$cycle.
\end{proof}

From Theorem  \ref{t777} we have the following result.

\begin{cor}
Let $R$ be a commutative ring. Then $Gr_{Z_2}(R(+)R)$ is connected if and only if $R$ is not simple if and only if $g(Gr_{Z_2}(R(+)R))=3$.
\end{cor}

Next we give a lower bound on the clique number of $Gr_{Z_2}(R(+)R)$ using the clique number of $R$.

\begin{thm}\label{231}
Let $R$ be a commutative ring.\\
(1) If $|G(R)|$ is infinite then so is $\omega(Gr_{Z_2}(R(+)R)$.\\
(2) If $|G(R)|$ is finite, then $\omega(Gr_{Z_2}(R(+)R)\geq 1+2\omega(G(R))+|G(R)|$ with equality holds if and only if $G(R)$ is null graph.
\end{thm}
\begin{proof}
Let $\mathcal{C}$ be a clique of maximal size in $G(R)$ and let $$\mathcal{H}_1=\{0(+)I\mid I\in \mathcal{C}\}$$ $$\mathcal{H}_2=\{I(+)I\mid I\in \mathcal{C}\}$$ $$\mathcal{H}_3=\{J(+)R\mid J\in I^{\ast}(R)\}.$$ Then $\mathcal{H}_1\cup \mathcal{H}_2\cup\mathcal{H}_3\cup\{0(+)R\}$ is a clique in $Gr_{Z_2}(R(+)R)$. \\
(1) If $|G(R)|$ is infinite then  $\omega(Gr_{Z_2}(R(+)R)$ is infinite because $|G(R)|=|\mathcal{H}_3|$.\\
(2) Assume $|G(R)|$ is finite. Then $|\mathcal{H}_1|=|\mathcal{H}_2|=\omega(G(R))$ and $|\mathcal{H}_3|=|G(R)|$. Consequently $\omega(Gr_{Z_2}(R(+)R)\geq 1+2\omega(G(R))+|G(R)|$. It is remaining to show the last part of (2). Assume $G(R)$ is not null graph. Then $|\mathcal{C}|\geq 2$. So we can pick $I,J\in \mathcal{C}$ such that $\{0\}\neq I\cap J \subsetneq I$. This implies that $\mathcal{H}_1\cup \mathcal{H}_2\cup\mathcal{H}_3\cup\{0(+)R\}\cup\{(I\cap J)(+)I\}$ is a clique in $Gr_{Z_2}(R(+)R)$, and so $\omega(Gr_{Z_2}(R(+)R)\geq 2+2\omega(G(R))+|G(R)|$. Conversely, assume $G(R)$ is a null graph. Then $g(G(R))=1$ and every ideal of $R$ is minimal as well as maximal. If $I(+)J$ is $Z_2-$graded ideal of $R(+)R$ then $RI\subseteq J$, and so $RI=0$, $I=RI=J$, or $J=R$. Moreover, if $I$ and $J$ are distinct proper ideals in $R$ then $(I(+)I)\cap (J(+)J)=\{(0,0)\}$. So for each proper ideal $I$ of $R$, $\{0(+)I, I(+)I, 0(+)R\}\cup\mathcal{H}_3$ is maximal clique in $R(+)R$. Hence equality holds.
\end{proof}

\begin{cor}
Let $R$ be a commutative ring. Then $Gr_{Z_2}(R(+)R)$ is planar if and only if  $R$ contains at most one proper nontrivial ideal
\end{cor}
\begin{proof}
If $|G(R)|\geq 2$. By Theorem \ref{231}, it follows that $K_5$ is a subgraph of $Gr_{Z_2}(R(+)R)$. So by Kuratowski's Theorem \cite[Theorem 9.10]{bo76}, $Gr_{Z_2}(R(+)R)$ is not planar. Conversely, Assume $R$ contains at most one proper nontrivial ideal. Then $|Gr_{Z_2}(R(+)R)|\leq 4$, and so it is planar.
\end{proof}

\section{$Gr_G(R)$ when $G$ is ordered group}\label{sec6}
An ordered group is a group $G$ together with a subset $S$ such that\\
(1) $e\notin S$\\
(2) If $\sigma\in G$, then $\sigma\in S$, $\sigma=e$, or $\sigma^{-1}\in S$\\
(3) If $\sigma,\tau\in S$ the $\sigma\tau\in S$\\
(4) $\sigma S \sigma^{-1}\subseteq S$, for all $\sigma\in G$\\

For $\sigma,\tau\in G $ we write $\sigma<\tau$ if and only if $\sigma^{-1}\tau\in S$ (equivalently $\tau\sigma^{-1}\in S$). Suppose that $R$ is $G-$graded ring where $G$ is an ordered group. Then any $r\in R$ can be written uniquely as $r=r_{\sigma_1}+r_{\sigma_2}+\ldots+r_{\sigma_n}$, with $\sigma_1<\sigma_2<\cdots<\sigma_n$. For each left ideal $I$ of $R$, denote by $I^{\sim}$ the graded ideal generated by the homogeneous components of highest degrees of all elements of $I$. We have the following result from \cite[Lemma 5.3.1, Corollary 5.3.3]{Nastasescue}

\begin{lem}\label{ll} Let $R$ be a $G-$graded ring where $G$ is ordered group. Then
\begin{enumerate}[(1)]
\item $I=I^{\sim}$ if and only if $I$ is $G-$graded left ideal.
\item\label{ll2} $I^{\sim}=\{0\}$ if and only if $I=\{0\}$.
\item\label{ll3} If $I\subseteq J$ then $I^{\sim}\subseteq J^{\sim}$.
\item\label{ll4} If $supp(R,G)$ is well ordered subset of $G$ and $I\subseteq J$ are left ideals then $I=J$ if and only if $I^{\sim}=J^{\sim}$
\end{enumerate}
\end{lem}

\begin{thm}\label{t543} Let $R$ be a $G-$graded ring where $G$ is an ordered group. If $supp(R,G)$ is well ordered subset of $G$ then $Gr_G(R)$ is connected if and only if $G(R)$ is connected.
\end{thm}
\begin{proof}
If $Gr_G(R)$ is connected then $G(R)$ is not null graph and therefore it is also connected. For the converse, assume that $G(R)$ is connected and let $I$ and $J$ be adjacent vertices of $G(R)$. Hence $I\cap J\neq\{0\}$. Let $K=I\cap J$. Since $I\neq J$ then either $K\subsetneq I$ or $K\subsetneq J$. Without loss of generality assume $K\subsetneq I$. Then by parts \ref{ll2}, \ref{ll3}, and \ref{ll4} of Lemma \ref{ll}, we have $\{0\}\neq K^{\sim}\subsetneq I^{\sim}$. So $Gr(R)$ is not null and hence it is connected.
\end{proof}

\begin{thm}\label{t544} Let $R$ be a $G-$graded where $G$ is an ordered group. If $supp(R,G)$ is well ordered subset of $G$ and $R$ is local ring then $g(Gr_G(R))=g(G(R))$.
\end{thm}

\begin{proof}
Clearly If $g(G(R))=\infty$ then $g(Gr_G(R))=\infty$. Assume that $g(G(R))<\infty$, it follows from Theorem \ref{t3} that $g(G(R))=3$. If $R$ is not left Noetherian, then we can find three nontrivial left ideals $I_1$, $I_2$, and $I_3$ such that $I_1\subsetneq I_2\subsetneq I_3$. Then, by part \ref{ll4} of Lemma \ref{ll}, we get that $I_1^{\sim}\subsetneq I_2^{\sim}\subsetneq I_3^{\sim}$. Hence $I_1^{\sim}-I_2^{\sim}-I_3^{\sim}$ is a $3-$cycle in $Gr_G(R)$. Now assume that $R$ is left Noetherian. This implies that $J\subseteq M$ for all $J\in I^{\ast}(R)$. Since $G(R)$ is not a star graph, there are two distinct left ideals $I,J\in I^{\ast}(R)\setminus\{M\}$ such that $I\cap J\neq\{0\}$. Without loss of generality, we may assume that $I\cap J\subsetneq I$. So we have $\{0\}\neq I\cap J\subsetneq I\subsetneq M$. Again by part \ref{ll4} of Lemma \ref{ll}, we obtain the $3-$cycle $(I\cap J)^{\sim}-I^{\sim}-M^{\sim}$ in $Gr_G(R)$. Therefore, $g(Gr_G(R))=3$. This completes the proof.
\end{proof}

\begin{rem}\label{r545} Take $R$ and $G$ as described in Theorem \ref{t544}, except for the condition ``R is local''. If $g(Gr_G(R))=\infty$, then we know from theorem \ref{t4} that  $R$ is a $G-$graded local ring. In this case, if and $g(G(R))=3$, then the followings hold:\\
(1) The unique $G-$graded maximal left ideal (Say $M$) is maximal among all proper left ideals.\\
(2) $K^{\sim}=M$ for every maximal ideal $K$ of $R$. \\
(3) The length of every acceding chain of left ideals is exactly four.\\
\end{rem}


\end{document}